\documentclass{amsart}

\usepackage{amssymb,amsmath,amsfonts,fancyhdr,amsthm}
\usepackage{graphicx}
\usepackage[dvips,all]{xy}
\UseComputerModernTips

\newtheorem{thm}{Theorem}[section]
\newtheorem{prop}[thm]{Proposition}
\newtheorem{lem}[thm]{Lemma}

\newtheorem{qst}[thm]{Question}

\newtheorem*{thrm1}{Theorem \ref{T1}}
\newtheorem*{thrm2}{Theorem \ref{T2}}

\theoremstyle{definition}

\newtheorem{eg}[thm]{Example}

\theoremstyle{remark}
\newtheorem{rmk}[thm]{Remark}

\newcommand{\Z}{\mathbb Z}

\renewcommand{\phi}{\varphi}

\renewcommand{\bar}{\protect\overline}

\renewcommand{\i}[1]{\mathfrak{#1}}

\newcommand{\q}{\i{q}}

\renewcommand{\a}{\i{a}}
\renewcommand{\b}{\i{b}}

\newcommand{\x}{\mathbf x}

\newcommand{\dd}{\mathbf d}

\newcommand{\depth}{\mathop{\mathrm{depth}}\nolimits}
\newcommand{\codim}{\mathop{\mathrm{codim}}\nolimits}

\newcommand{\reg}{\mathop{\mathrm{reg}}\nolimits}

\newcommand{\pd}{\mathop{\mathrm{pd}}\nolimits}

\newcommand{\tor}{\mathop{\mathrm{Tor}}\nolimits}

\title[A Polynomial Bound on Regularity]{A Polynomial Bound on the Regularity of an Ideal in Terms of Half of the Syzygies}
\author{Jason M\lowercase{c}Cullough}
\address{Department of Mathematics, Univeristy of California, Riverside, 900 University Ave., Riverside, CA 92521}
\email{jmccullo@math.ucr.edu}
\subjclass[2010]{Primary: 13D02; Secondary: 13D07, 13F20}

\keywords{regularity, Betti numbers, resolution}

\begin{document}

\begin{abstract} Let $K$ be a field and let $S = K[x_1,\ldots,x_n]$ be a polynomial ring.  Consider a homogenous ideal $I \subset S$.  Let $t_i$ denote $\reg(\tor_i^S(S/I,K))$, the maximal degree of an $i$th syzygy of $S/I$.  We prove bounds on the numbers $t_i$ for $i > \lceil \frac{n}{2} \rceil$ purely in terms of the previous $t_i$.  As a result, we give bounds on the regularity of $S/I$ in terms of as few as half of the numbers $t_i$.  We also prove related bounds for arbitrary modules.  These bounds are often much smaller than the known doubly exponential bound on regularity purely in terms of $t_1$.  
\end{abstract}
\maketitle

\section{Introduction}

Given a homogeneous ideal $I$ of $S = K[x_1,\ldots,x_n]$, it is natural to seek bounds on and relations among the degrees of the syzygies of $S/I$.  Doing so can yield interesting bounds on the regularity of $S/I$.   We set $t_i = \reg(\tor_i^S(S/I,K))$.  Then $t_i$ is the maximal degree of a minimal $i$th syzygy of $S/I$.  There are several bounds on the regularity of $S/I$ (or equivalently, on $I$), purely in terms of the degrees of the generators of $I$.  Note that the maximal degree of a minimal generator of $I$ is just $t_1(S/I)$.  All of these bounds are doubly exponential in terms of $t_1(S/I)$.  Examples of Mayr and Meyer \cite{MM} show that we cannot avoid this doubly exponential behavior.  It seems reasonable, however, that given more information about the resolution and the degrees of the syzygies of $S/I$, better bounds should be possible.  Indeed, Engheta asks in \cite{EnghetaT} if there is a polynomial bound on the regularity of an ideal if, for some $k \ge 2$, the degrees of the minimal generators of the first 
$k$ syzygies are known as well?  We give two results of this flavor.

\begin{thrm2}
Let $I \subset S = K[x_1,\ldots,x_n]$ be a homogeneous ideal.  Set $h = \lceil \frac{n}{2} \rceil$ and let $t_i = t_i(S/I)$.  Then
\[\reg(S/I) \le \sum_{i = 1}^h t_i + \frac{\prod_{i = 1}^h t_i}{(h-1)!}.\]
 \end{thrm2}

Hence we achieve a bound on the regularity of $S/I$ in terms of only half of the degrees of the syzygies.  Note that this shows that if the numbers $t_i$ are not doubly exponential in terms of $t_1$ in the first half of the resolution, they cannot be doubly exponential in the second half. 

We also prove the following bound, which shows that the final degree jump in the resolution of $S/I$ cannot be large relative to the preceding jumps.
\begin{thrm1} Let $I \subset S$ be a homogenous ideal.  Set $p = \pd(S/I)$ and $t_i = t_i(S/I)$.  Then
\[t_p \le \max\{ t_i + t_{p-i}\,|\, i = 1,\ldots,p-1\}.\]
In particular,
\[\reg(S/I) \le \max\{t_i + t_{p-i} - p\,|\, i = 1,\ldots,p-1\}.\]
\end{thrm1}
\noindent For $n \le 3$, this recovers the low-dimensional case of a result of Eisenbud-Huneke-Ulrich in \cite{EHU}.

 These results reveal interesting restrictions on the possible Betti diagrams of cyclic modules $S/I$.  
 We also give slightly more general bound for arbitrary modules.
 Our methods involve a careful analysis of the Boij-S\"oderberg numerics in the decomposition of the Betti table into a positive rational sum of pure diagrams.

The rest of the paper is structured as follows: In Section 2, we set notation and summarize related results.  In Section 3, we give a short review of Boij-S\"oderberg theory.  Section 4 contains the main inequality we need to prove our main theorems.  We close in Section 5 with some examples and questions about similar bounds.

\section{Background and Terminology}

We now fix notation for the remainder of the paper.  Let $K$ denote an arbitrary field and let $S = K[x_1,\ldots,x_n]$ denote a polynomial ring over $K$.  We consider $S = \bigoplus_{i = 0}^\infty S_i$ as a graded ring with the standard grading.  For $d \in \Z$, let $S(-d)$ denote the rank one free $S$-module whose generator is in degree $d$.  In other words, the $i$th graded part of $S(-d)$ is $S(-d)_i = S_{i-d}$.  Given any finitely generated graded $S$-module $M$, we form the minimal graded free resolution
\[0 \to \bigoplus_j S(-j)^{\beta_{p,j}(M)} \to \cdots \to \bigoplus_j S(-j)^{\beta_{1,j}(M)} \to \bigoplus_j S(-j)^{\beta_{0,j}(M)} \to M \to 0.\]
The integers $\beta_{i,j}(M)$ are called the \textsf{Betti numbers} of $M$ and are commonly displayed in a matrix called the \textsf{Betti diagram}:\\
\[ \begin{tabular}{c|ccccc}
      &0&1&$\cdots$ &$i$&$\cdots$ \\ 
      \hline \text{0:}&$\beta_{0,0}(M)$ &$\beta_{1,1}(M)$ &$\cdots$&$\beta_{i,i}(M)$&$\cdots$\\
       \text{1:}&$\beta_{0,1}(M)$ &$\beta_{1,2}(M)$ &$\cdots$&$\beta_{i,i+1}(M)$&$\cdots$\\
       $\vdots$&$\vdots$ &$\vdots$ &$$&$\vdots$&$$\\
            $j$:&$\beta_{0,j}(M)$ &$\beta_{1,j}(M)$ &$\cdots$&$\beta_{i,i+j}(M)$&$\cdots$\\
                   $\vdots$&$\vdots$ &$\vdots$ &$$&$\vdots$&$$\\
      \end{tabular}
      \]
We then define two measures of the complexity of $M$.  The \textsf{projective dimension} of $M$ is $\pd(M) := \max\{i\,|\,\beta_{i,j}(M) \neq 0 \text{ for some $j$}\}$.  The \textsf{Castelnuovo-Mumford regularity} of $M$ (or just regularity of $M$) is $\reg(M) := \max \{j\,|\,\beta_{i,i+j}(M) \neq 0 \text{ for some $i$}\}.$  We set 
\[t_i(M) := \reg( \tor_i^S(M,K) ) = \max\{j\,|\,\beta_{i,j}(M) \neq 0\}.\]
Note that regularity could be defined as $\reg(M) = \max\{t_i(M) - i\,|\,0 \le i \le \pd(M)\}.$

Eisenbud, Huneke and Ulrich proved the following weak convexity inequality on the degrees of the syzygies of a cyclic module:
\begin{thm}[\cite{EHU} Corollary 4.1]\label{ehu1} If $\dim S/I \le 1$, then
\[t_n(S/I) \le t_i(S/I) + t_{n-i}(S/I).\]
\end{thm}
They also show that under nice hypotheses on $I$, a similar inequality holds.
\begin{thm}[\cite{EHU} Corollary 4.2]\label{ehu2} If $c = \codim(S/I)$ and $\delta := \dim(S/I) - \depth(S/I) \le 1$, and $I$ contains a regular sequence of forms of degrees $d_1,\ldots,d_q$, then
\[t_{c + \delta}(S/I) \le t_{c + \delta - q}(S/I) + d_1 + \cdots d_q.\]
In particular, if $S/I$ is Cohen-Macaulay of codimension $c$ and $I$ is generated in degree $d$, then
\[t_c \le t_{c-q} + dq.\]
\end{thm}
However, if $S/I$ is not Cohen-Macaulay, the resolution of $S/I$ may not be so well-behaved.  The authors observed that their hypotheses on $S/I$ in Theorem~\ref{ehu2} are necessary in light of an example by Caviglia \cite{Caviglia}.  He defined a family of three-generated ideals with quadratically growing regularity relative to the degrees of the generators.  In fact, he defined ideals $I_r$ in $S = K[x_1,x_2,x_3,x_4]$ for $n \ge 2$ with $t_1(S/I_r) = r$ and $t_2(S/I_r) = r^2$, so that $\reg(S/I) = r^2 - 2 = t_1(S/I)^2 - 2$.  These ideals have codimension $2$ and depth $0$, so the result above does not apply.  Whether Theorem~\ref{ehu1} holds in greater generality is less clear and we address this question in Section 5.

The most general result bounding regularity in terms of the degrees of some of the syzygies is the following result, due in characteristic 0 to Galligo \cite{Galligo1}, \cite{Galligo2} and Giusti \cite{Giusti}, and later in all characteristics by Caviglia and Sbarra \cite{CS}.
\begin{thm}\label{Texp} 
Let $I \subset K[x_1,\ldots,x_n] = S$ be an ideal generated in degree $\le d$.  (So that $t_1(S/I) \le d$.)  Then
\[ \reg(I) \le (2d)^{2^{n-2}}.\]
\end{thm}
Mayr and Meyer \cite{MM} produced a family of examples of homogeneous ideals $\mathfrak{a}_r \subset K[x_1,\ldots, x_r] $ for which the ideal membership problem had doubly exponential complexity in terms of the degrees of the generators.  Bayer and Stillman \cite{BS} showed that these ideals also had doubly exponential regularity, which was exhibited in the first syzygies of $\mathfrak{a}_r$.  In other words $t_0(\mathfrak{a}_r) = 4$ and $t_1(\mathfrak{a}_r) \ge 2^{2^{(r-2)/10}}$.   Thus we cannot hope to avoid doubly exponential behavior given only the degrees of the generators.  

It is striking, however, that the examples of ``wild'' regularity, such as the Mayr-Meyer ideals, Caviglia's examples, or those in \cite{BMNSSS}, all have large regularity by the first syzygies of $I$ (or second syzygies of $S/I$).  The purpose of this paper is to give a bound on the degrees of the later syzygies in a resolution in terms of the earlier ones.  In turn, this yields interesting regularity bounds for cyclic modules $S/I$.  This serves to give some indication why the examples mentioned have large regularity early in the resolution.  While our bounds are much larger than the bounds of Eisenbud, Huneke and Ulrich and are likely not tight, they hold without any hypotheses on $I$.  And while our result requires more data on the syzygies of $S/I$ than just the degrees of the minimal generators of $I$, it provides much smaller bounds than the doubly exponential bound above.  Our main technical tool in proving these results is the numerics resulting from Boij-S\"oderberg decomposition \cite{BS1}.

In \cite{Erman}, Erman used Boij-S\"oderberg numerics to prove a special case of the Buchsbaum-Eisenbud-Horrocks conjecture on the ranks of free modules appearing in a free resolution.  He used restrictions imposed by the Boij-S\"oderberg decomposition to show that ideals with small regularity relative to the degrees of the generators of $I$ satisfy the conjecture.  While our methods are similar and the ideas here were inspired by Erman's techniques, we do not use his results directly.

\section{Review of Boij-S\"oderberg Theory}

We follow the  notation in \cite{Floystad}.  We say that a sequence $\dd = (d_0,\ldots,d_s) \in \Z^{s+1}$ is a \textsf{degree sequence} (of length $s+1$) if $d_{i-1} < d_{i}$ for $i = 1,\ldots,s$.  Define $\Z^{s+1}_{\deg}$ to be the set of degree sequences of length $s+1$.  Given two degree sequences $\dd,\dd' \in \Z^{s+1}_{\deg}$, we say $\dd \le \dd'$ if $d_i \le d_i'$ for $i = 0,\ldots,s$.  For $\a,\b \in \Z^{s+1}_{\deg}$ with $\a \le \b$, we set $\mathbb{D}(\a,\b) := \{ \dd \in \Z^{s+1}_{\deg}\,|\, \a \le \dd \le \b\}$.  If $\dd = (d_0,\ldots,d_p) \in \Z^{p+1}_{\deg}$ and $s \le p$, then we set $\tau_s(\dd) = (d_0,\ldots,d_s)$.

A graded $S$-module $M$ is called \textsf{pure} of type $\dd = (d_0,\ldots,d_s)$, if 
$\beta_{i,j}(M) \neq 0$ if and only if $j = d_i$ for $i = 0,\ldots,s$.  Hence a pure module has a graded free resolution of the form
\[0 \to S(-d_s)^{\beta_{s,d_s}} \to \cdots \to S(-d_1)^{\beta_{1,d_1}} \to S(-d_0)^{\beta_{0,d_0}} \to M \to 0.\]
In \cite{HK}, Herzog and Kuhl showed that any graded pure Cohen-Macaulay $S$-module has prescribed Betti numbers up to constant multiple.  Each degree sequence $\dd = (d_0,\ldots,d_s)$ then defines a ray in the cone of Betti diagrams and there is a unique point $\bar{\pi}(\dd)$ on this ray with $\beta_{0,d_0}(\bar{\pi}(\dd)) = 1$.  In particular, there are specific formulas, called the \textsf{Herzog-Kuhl equations}, for the Betti numbers appearing in $\bar{\pi}(\dd)$; namely,
\[\beta_i(\dd) := \beta_{i,d_i}(\bar{\pi}(\dd)) = \prod_{\overset{1 \le j \le s}{j \neq i}} \frac{|d_j - d_0|}{|d_j - d_i|}.\]
So for example, we have
\[\bar{\pi}(0,2,3,6) = \begin{tabular}{r|rrrr}
      &0&1&2&3\\ \hline \text{0:}&$1$&\text{.}&\text{.}&\text{.}\\\text{
      1:}&\text{.}&$\frac{9}{2}$&$4$&\text{.}\\\text{2:}&\text{.}&\text{.}&\text{.}&\text{.}\\\text
      {3:}&\text{.}&\text{.}&\text{.}&$\frac{1}{2}$\\\end{tabular}
\]
We will use $\beta_i(\dd)$ to denote the above formula even when the integers $d_j$ do not necessarily form a strictly increasing sequence.

Now let $M$ be a graded $S$-module.  Set $p = \pd(M)$ and $c = \codim(M)$.  For $i = 0,\ldots,p$, we define $\bar{d}_i = \max\{ j\,|\,\beta_{i,j}(M) \neq 0\}$ and $\underline{d}_i = \min\{ j \,|\,\beta_{i,j}(M) \neq 0\}$ and then set $\underline{\dd}(M) = (\underline{d}_0,\ldots,\underline{d}_p)$ and $\bar{\dd}(M) = (\bar{d}_0,\ldots,\bar{d}_p)$.

Eisenbud and Schreyer showed that the Betti diagram of any graded Cohen-Macaulay $S$-module $M$ is a positive rational sum of pure diagrams.  Boij and S\"oderberg extended this result to the non-Cohen-Macaulay case.  Here we state a version of their theorem that will suffice for the purposes of this paper.
\begin{thm}(\cite{ES},\cite{BS2}) Let $M$ be a graded $S$-module of projective dimension $p$ and codimension $c$.  Then the Betti diagram $\beta(M)$ can be decomposed as a sum:

\[\beta(M) = \sum_{c \le s \le p} \,\,\sum_{\dd \in \mathbb{D}(\tau_s(\underline{\dd}(M)),\tau_s(\bar{\dd}(M)))} q_{\dd} \beta(\bar{\pi}(\dd)),\]
where the $\q_\dd$ are nonnegative rational numbers.
\end{thm}
A much stronger statement is possible yielding a unique decomposition on the right given a saturated chain of degree sequences between $\tau_s(\underline{\dd}(M))$ to $\tau_s(\bar{\dd}(M))$ for $c \le s \le p$.  This stronger result also provides an algorithm for producing such a decomposition of the Betti diagram of any graded $S$-module.  However, we will not need these stronger statements and refer the interested reader to \cite{BS2}, \cite{BS1}, \cite{ES}.

\section{Main Results}

In this section we show how one can use the Boij-S\"oderberg decomposition of the Betti table of a grade $S$-module to produce bounds on regularity.  We start with a simple lemma.

\begin{lem}\label{L1} Let $\a = (a_0,\ldots,a_s)$ and $\b = (b_0,\ldots,b_s)$ be degree sequences  of length $s+1$.  Suppose $a_0 \ge b_0$, $a_s = b_s$, and $a_i \le b_i$ for all $i = 1,\ldots, s-1$.  Then
\[\beta_s(\a) \le \beta_s(\b).\]
\end{lem}

\begin{proof} Since $a_i \le b_i$, we have $a_s - a_i \ge a_s - b_i = b_s - b_i$ for $i = 1,\ldots, s-1$.  Similarly $a_i - a_0 \le b_i - a_0 \le b_i - b_0$ for $i = 1,\ldots, s-1$.  Therefore 
\[\beta_s(\a) = \prod_{i = 1}^{s-1} \frac{a_i - a_0}{a_s - a_i} \le \prod_{i = 1}^{s-1} \frac{b_i - b_0}{b_s - b_i} = \beta_s(\b).\]
\end{proof}

The next result contains the main idea needed for all of the subsequent bounds.  It follows by noticing that if we fix at least half of the degrees in a degree sequence $\dd$, then the final Betti number in $\beta(\bar{\pi}(\dd))$ tends to $0$ as the regularity increases.

\begin{prop}\label{P1} Let $M$ be a finitely generated graded $S$-module.  Let $p = \pd(M)$, $t_i = t_i(M)$ $\mu(M) = \sum_j \beta_{0,j}$, and $\underline{d}_0 = \min\{j \,|\, \beta_{0,j}(M) \neq 0\}$.  Fix an integer $h < p$ and let $B \ge \max \{t_i - i\,|\, i = 1,\ldots,h\}$.  Suppose for all integers $r$ and $s$ with $h < s \le p$ and $r > B$, we have
\[\beta_s(\underline{d}_0, t_1, t_2, \ldots, t_h, r+h+1, r+h+2,\ldots, r+s) < \frac{1}{\mu(M)}.\]
Then 
 \[\reg(M) \le B.\]
\end{prop}

\begin{proof} Suppose $\reg(S/I) > B$.  Since $B \ge t_i - i$ for all $i = 1,\ldots,h$, there exist integers $i > h$ and $j > B + i$ with $\beta_{i,j}(M) \neq 0$.  Fix $i'$ and then $j'$ maximal among those ordered pairs $(i',j')$ with $i' > h$ and $j' > B + i'$ with $\beta_{i',j'}(M) \neq 0$.  
Now consider the Boij-S\"oderberg decomposition of $\beta(M)$
\[\beta(M) = \sum_{c \le s \le p} \,\,\sum_{\dd \in \mathbb{D}(\tau_s(\underline{\dd}(M)),\tau_s(\bar{\dd}(M)))} q_{\dd} \beta(\bar{\pi}(\dd)).\]
Since $\mu(M) = \sum_j \beta_{0,j}$, it follows that the rational coefficients $q_\dd$ sum to $\mu(M)$; that is,
\[\sum_{c \le s \le p} \,\,\sum_{\dd \in \mathbb{D}(\tau_s(\underline{\dd}(M)),\tau_s(\bar{\dd}(M)))} q_{\dd} = \mu(M).\]

Now consider only those degree sequences $\dd \in \mathbb{D}(\tau_{i'}(\underline{\dd}(M)),\tau_{i'}(\bar{\dd}(M)))$ with $d_{i'} = j'$.  By our choice of $i', j'$, these degree sequences correspond to the only pure diagrams represented in the sum with nonzero entry in the $(i',j')$ coordinate.  By the previous lemma, 
\[\beta_{i'}(\dd) \le \beta_{i'}(d, t_1, t_2, \ldots, t_h, r+h+1, r+h+2,\ldots, r+i') < \frac{1}{\mu(M)}\]
for all such degree sequences $\dd$, where $r = j' - i' > B$.  Hence, 
\begin{align*}
\beta_{i',j'}(M) &= \sum_{\overset{\dd \in \mathbb{D}(\tau_{i'}(\underline{\dd}(M)),\tau_{i'}(\bar{\dd}(M)))}{d_{i'} = j'}} q_{\dd} \beta_{i'}(\dd) \\
&<  \frac{1}{\mu(M)} \sum_{\overset{\dd \in \mathbb{D}(\tau_{i'}(\underline{\dd}(M)),\tau_{i'}(\bar{\dd}(M)))}{d_{i'} = j'}} q_{\dd} \\
& \le \frac{1}{\mu(M)} \mu(M) = 1,
\end{align*}
which is impossible.
\end{proof}

In the following we give two cases where the previous proposition yields interesting bounds on the regularity of cyclic modules.  First we take the case $h = p - 1$.

\begin{lem}\label{L2} Let $I \subset S$.  Set $p = \pd(S/I)$ and $t_i = t_i(S/I)$.  For any integer $r$ with
\[r >  \max\{ t_i + t_{p-i}\,|\, i = 1,\ldots,p-1\},\]
we have that
\[\beta_p(\dd) < 1,\]
for all degree sequences $\dd$ of length $p+1$ with
\[  (0,0,\ldots,0,r) \le \dd \le (0,t_1,t_2,\ldots,t_{p-1}, r).\]
\end{lem}

\begin{proof}
By Lemma~\ref{L1}, it suffices to check that $\beta_p(0,t_1,t_2,\ldots,t_{p-1},r) < 1$.  Note that $r > t_i + t_{p-i}$ for $i = 1,\ldots,p-1$.  Hence $r - t_i > t_{p-i}$ for $i = 1,\ldots,p-1$.  Therefore
\[\beta_p(0,t_1,t_2,\ldots,t_{p-1},r) = \frac{\prod_{i = 1}^{p-1} t_i}{\prod_{i = 1}^{p-1} (r - t_i)} < \frac{\prod_{i = 1}^{p-1} t_i}{\prod_{i = 1}^{p-1} t_{p - i}} = 1.\]
\end{proof}

We thus have the following bound on the degree of the final syzygies of $S/I$.

\begin{thm}\label{T1} Let $I \subset S$ be a homogenous ideal.  Set $p = \pd(S/I)$ and $t_i = t_i(S/I)$.  Then
\[t_p \le \max\{ t_i + t_{p-i}\,|\, i = 1,\ldots,p-1\}.\]
In particular,
\[\reg(S/I) \le \max\{t_i + t_{p-i} - p\,|\, i = 1,\ldots,p-1\}.\]
\end{thm}

\begin{proof} Since $\mu(S/I) = 1$, this follows immediately from Lemma~\ref{L2} and Proposition~\ref{P1}.
\end{proof}

In the following section we give some examples and compare this bound with known results.  For the next result, we will need the following fact, whose proof we leave to the reader.

\begin{lem} Let $i < h$ be positive integers and let $\sigma_i(\x)$ denote the $i$th elementary symmetric polynomial in $\x = x_1, x_2,\ldots,x_h$.  Then for any nonnegative real numbers $\dd = (d_1, \ldots, d_h)$, we have
\[\sigma_i(\dd) \sigma_1(\dd) \ge \sigma_{i + 1}(\dd).\]
\end{lem}

We now show that given as few as half of the $t_i$, we can give a bound on the regularity of $S/I$.

\begin{lem}\label{L3}
Let $I \subset S$ be a homogeneous ideal.  Set $p = \pd(S/I)$ and $t_i = t_i(S/I)$.  Suppose $h \ge \lceil \frac{p}{2} \rceil$.  For any integers $r$ and $s$ with $h \le s \le p$ and
\[r > \sum_{i = 1}^h t_i + \frac{\prod_{i = 1}^h t_i}{(h-1)!} - h,\]
we have that
\[\beta_s(0, t_1, t_2,\ldots, t_h, r+h+1, r+h+2,\ldots, r+s) < 1.\]
\end{lem}

\begin{proof} Again by Lemma~\ref{L1} it suffices to check that $\beta_s(\dd) < 1$ for 
\[\dd = (0, d_1, d_2,\ldots, d_h, r+h+1, r+h+2,\ldots, r+s).\]
Set $R = r + s$.  We have
\begin{align*}\beta_s(\dd)   &= \frac{\prod_{i = 1}^{h} d_i \prod_{j = h+1}^{s-1} (r + j)}{\prod_{i = 1}^{h} (R - d_i) \prod_{j = h+1}^{s-1} (R - (r + j))}\\
&\le  \frac{\left(\prod_{i = 1}^{h} d_i\right) R^{s - h-1}}{\left(\prod_{i = 1}^{h} (R - d_i)\right) (s - h - 1)!}\\
&\le  \frac{\left(\prod_{i = 1}^{h} d_i\right) R^{h-1}}{\left(\prod_{i = 1}^{h} (R - d_i)\right) (h - 1)!}\\
\end{align*}
where the last inequality holds because 
\[s - h - 1 \le p - h - 1 \le p - \frac{p}{2} - 1 \le \left\lceil \frac{p}{2} \right\rceil -1 = h - 1.\]

The quantity above is less then $1$ if and only if
\[\prod_{i = 1}^{h} (R - d_i) - \frac{\prod_{i = 1}^{h} d_i}{(h-1)!} R^{h-1}  > 0.\]
We rewrite this as
\[  \prod_{i = 1}^{h} (R - d_i) - \frac{\prod_{i = 1}^{h} d_i}{(h-1)!} R^{h-1}   = R^h - \left(\sum_{i = 1}^h d_i + \frac{\prod_{i = 1}^{h} d_i}{(h-1)!}\right) R^{h-1} + \sum_{i = 2}^h \sigma_i(\dd) (-1)^i R^{h-i},\]
where $\sigma_i(\x)$ denotes the $i$th elementary symmetric polynomial on $h$ variables.
Since 
\[R  = r+s > \sum_{i = 1}^h t_i + \frac{\prod_{i = 1}^h t_i}{(h-1)!},\]
it follows that 
\[R^h - \left(\sum_{i = 1}^h d_i + \frac{\prod_{i = 1}^{h} d_i}{(h-1)!}\right) R^{h-1} = R^{h-1} \left(R - \left(\sum_{i = 1}^h d_i + \frac{\prod_{i = 1}^{h} d_i}{(h-1)!}\right) \right) > 0.\]

Now fix $i$ even with $2 \le i < h$.  Then we have
\begin{align*}(-1)^i \sigma_i(\dd) R^{h-i} + (-1)^{i+1}\sigma_{i+1}(\dd) R^{h-(i+1)} &=
\sigma_i(\dd) R^{h-i} - \sigma_{i+1}(\dd) R^{h - i - 1} \\
&= R^{h-i-1} \left(\sigma_i(\dd) R - \sigma_{i+1}(\dd) \right)\\
&\ge R^{h-i-1} \left(\sigma_i(\dd) \sigma_1(\dd) - \sigma_{i+1}(\dd)\right).\\
\end{align*}
Note that the last inequality follows, since
\[R \ge \sum_{i = 1}^h d_i.\]
The final term
\[R^{h-i-1} \left(\sigma_i(\dd) \sigma_1(\dd) - \sigma_{i+1}(\dd)\right)\]
is nonnegative since 
\[\sigma_i(\dd) \sigma_1(\dd) \ge \sigma_{i+1}(\dd),\]
by the previous lemma.  Hence by pairing off terms in the sum $\sum_{i = 2}^h (-1)^i \sigma_i(\dd) R^{h-i}$ (possibly leaving the term $(-1)^h \sigma_h(\dd)$ unpaired and positive in the case where $h$ is even), we see that it is also positive, which finishes the argument.
\end{proof}

\begin{thm}\label{T2}
Let $I \subset S = K[x_1,\ldots,x_n]$ be a homogeneous ideal.  Set $h = \lceil \frac{n}{2} \rceil$ and let $t_i = t_i(S/I)$.  Then
\[\reg(S/I) \le \sum_{i = 1}^h t_i + \frac{\prod_{i = 1}^h t_i}{(h-1)!}.\]
 \end{thm}

\begin{proof} This follows from Lemma~\ref{L3} and Proposition~\ref{P1}.
\end{proof}

\begin{rmk} It is worth noting that the previous bound does not hold in general for non-cyclic modules.  Constructions in \cite{EFW}, \cite{ES}, and \cite{BEKS} show that for any degree sequence $\dd = (d_0, d_1,\ldots, d_s)$, there exists a Cohen-Macaulay graded $S$-module $M$ with pure free resolution with degrees corresponding to $\dd$.  Hence no matter what degrees we pick for the first half of the resolution, no bound like this is possible in general. 
\end{rmk}

\section{Examples and Questions}

Let $I \subset S = K[x_1,x_2,x_3]$ be a homogenous ideal.  Theorem~\ref{T1} shows $t_3(S/I) \le t_1(S/I) + t_2(S/I).$  Hence $\reg(S/I) \le t_i + t_{3-i} - 3$ for $i = 0,\ldots,3$.  This slightly extends Theorem~\ref{ehu1} of Eisenbud-Huneke-Ulrich \cite{EHU} in the case $\dim(S) = 3$ by removing the hypothesis on $\dim(S/I)$.  Therefore we ask the following question.

\begin{qst} Given a homogenous ideal $I \subset S = K[x_1,\ldots,x_n]$ and $0 \le i \le n$, is
\[t_n(S/I) \le t_i(S/I) + t_{n-i}(S/I)\]
without any restriction the dimension of $S/I$?  
\end{qst}

\noindent We know of no counterexamples to this question, yet the proof in \cite{EHU} seems to require the $\dim(S/I) \le 1$ hypothesis.  

One might hope that the above statement $t_3(S/I) \le t_1(S/I) + t_2(S/I)$ holds without a restriction on the number of variables.  The following example shows that this is not the case.

\begin{eg} Let $R = K[x_1,\ldots,x_8]$ and let 
\[I = (x_1^6, x_2^6,x_1 x_3^2 + x_2 x_4^2, x_1 x_5^5 + x_4 x_6^5, x_1 x_7^5 + x_6 x_8^5).\]  Using Macaulay2 \cite{M2}, we compute the Betti table of the free resolution of $R/I$ below:

\[\begin{tabular}{r|ccccccc}
      &0&1&2&3&4&5&6 \\ \hline \text{total:}&1&5&19&46&60&39&10 \\ \hline \text{0:}&1&\text{-}&\text{-}&\text{-}&\text{-}&\text{-}&\text{-}\\\text{1:}&\text{-}&\text{-}&\text
      {-}&\text{-}&\text{-}&\text{-}&\text{-}\\\text{2:}&\text{-}&1&\text{-}&\text
      {-}&\text{-}&\text{-}&\text{-}\\\text{3:}&\text{-}&\text{-}&\text{-}&\text
      {-}&\text{-}&\text{-}&\text{-}\\\text{4:}&\text{-}&\text{-}&\text{-}&\text
      {-}&\text{-}&\text{-}&\text{-}\\\text{5:}&\text{-}&4&\text{-}&\text{-}&\text
      {-}&\text{-}&\text{-}\\\text{6:}&\text{-}&\text{-}&\text{-}&\text{-}&\text
      {-}&\text{-}&\text{-}\\\text{7:}&\text{-}&\text{-}&4&\text{-}&\text{-}&\text
      {-}&\text{-}\\\text{8:}&\text{-}&\text{-}&\text{-}&\text{-}&\text{-}&\text
      {-}&\text{-}\\\text{9:}&\text{-}&\text{-}&\text{-}&\text{-}&\text{-}&\text
      {-}&\text{-}\\\text{10:}&\text{-}&\text{-}&6&\text{-}&\text{-}&\text{-}&\text
      {-}\\\text{11:}&\text{-}&\text{-}&2&4&\text{-}&\text{-}&\text{-}\\\text{12:}&
      \text{-}&\text{-}&1&6&1&\text{-}&\text{-}\\\text{13:}&\text{-}&\text{-}&1&2&1&
      \text{-}&\text{-}\\\text{14:}&\text{-}&\text{-}&2&4&1&\text{-}&\text{-}\\\text{
      15:}&\text{-}&\text{-}&1&7&3&\text{-}&\text{-}\\\text{16:}&\text{-}&\text{-}&2
      &7&10&1&\text{-}\\\text{17:}&\text{-}&\text{-}&\text{-}&2&4&2&\text{-}\\\text{
      18:}&\text{-}&\text{-}&\text{-}&2&4&2&\text{-}\\\text{19:}&\text{-}&\text{-}&
      \text{-}&3&7&3&\text{-}\\\text{20:}&\text{-}&\text{-}&\text{-}&3&8&7&1\\\text{
      21:}&\text{-}&\text{-}&\text{-}&2&6&6&2\\\text{22:}&\text{-}&\text{-}&\text
      {-}&1&4&5&2\\\text{23:}&\text{-}&\text{-}&\text{-}&1&4&5&2\\\text{24:}&\text
      {-}&\text{-}&\text{-}&1&4&5&2\\\text{25:}&\text{-}&\text{-}&\text{-}&1&3&3&1
      \\\end{tabular}
      \]
\\
\noindent Hence we have that $t_3 = 28 > 24 = 6 + 18 = t_1 + t_2$.  And $\reg(R/I) = 25 > t_1 + t_2 -  3$.  We note however that the large regularity jumps happen early in the resolution. 
\end{eg}

\begin{eg}
The examples by Caviglia show that even with only four variables, $t_2(S/I)$ can grow quadratically with respect to $t_1(S/I)$.  The Theorem~\ref{T2} with $n = 4$ and $h = 2$ shows that for any homogeneous ideal $I \subset S = K[x_1, x_2, x_3, x_4]$, we have
\[\reg(S/I) \le t_1 + t_2 + t_1\cdot t_2.\]
Hence $t_3(S/I)$ and $t_4(S/I)$ cannot grow quadratically purely in terms of $t_2$.  
\end{eg}

Finally we remark that better bounds on $\reg(S/I)$ are sometimes possible by applying Proposition~\ref{P1} to $I$ in the case where $I$ has only a few generators.   We use the case of a three-generated ideal as an example.

\begin{eg}
Let $I \subset S = K[x_1,\ldots, x_5]$ be a homogeneous ideal with three degree $11$ generators and suppose that $t_2(S/I) = 12$ and $t_3(S/I) = 13$.  Then Theorem~\ref{T2} shows that
\[
\reg(S/I) \le 11 + 12 + 13 + \frac{11\cdot 12 \cdot 13}{2} = 894.
\]
Now consider the Betti diagram of $I$ instead of $S/I$.  
Since $\dim(S) = 5$, we have $\pd(I) \le 4$.  We observe that if $r > 15$, then both
\[\beta_3(11,12,13,r+3) = \frac{1 \cdot 2}{(r - 9)(r - 10)} < \frac{1}{3},\]
and
\[\beta_4(11,12,13,r+3,r+4) = \frac{1 \cdot 2 \cdot (r - 8)}{(r - 8)(r - 9)} < \frac{1}{3}.\]
By Proposition~\ref{P1}, we have that $\reg(S/I) = \reg(I) -1 \le 15 - 1 = 14$.
Clearly this method will not work as well if $I$ has many minimal generators.
\end{eg}

We close by noting that Theorem~\ref{T2} provides a polynomial bound on the regularity of $S/I$ given half of the syzygies.  It would be interesting to know if a different polynomial bound is possible using only the first $k < \left\lceil\frac{n}{2}\right\rceil$ syzygies.  Clearly for $k = 1$ this is not possible, but it is not clear what is possible for $2 \le k < \left\lceil\frac{n}{2}\right\rceil$.

\section{Acknowledgements}

\bibliographystyle{amsplain}
\bibliography{bib}

\end{document}